\documentclass{amsart}

\usepackage{amssymb}

\usepackage{tikz}

\theoremstyle{plain}
\newtheorem{theorem}{Theorem}
\newtheorem{lemma}[theorem]{Lemma}
\newtheorem{proposition}[theorem]{Proposition}

\theoremstyle{remark}
\newtheorem{remark}[theorem]{Remark}
\newtheorem{claim}[theorem]{Claim}
\newtheorem*{acknowledgments}{Acknowledgments}

\tikzset{myarrow/.style={->,>=stealth,line width=0.5pt}}

\begin{document}

\title{Entire or rational maps with integer multipliers}

\author{Xavier Buff}
\address{Institut de Math\'{e}matiques de Toulouse, UMR 5219, Universit\'{e} de Toulouse, CNRS, UPS, F-31062 Toulouse Cedex 9, France}
\email{xavier.buff@math.univ-toulouse.fr}

\author{Thomas Gauthier}
\address{Laboratoire de Math\'{e}matiques d'Orsay, UMR 8628, Universit\'{e} Paris-Saclay, CNRS, F-91405 Orsay Cedex, France}
\email{thomas.gauthier1@universite-paris-saclay.fr}

\author{Valentin Huguin}
\address{Constructor University Bremen gGmbH, Campus Ring 1, 28759 Bremen, Germany}
\email{vhuguin@constructor.university}

\author{Jasmin Raissy}
\address{Institut de Math\'{e}matiques de Bordeaux, UMR 5251, Universit\'{e} de Bordeaux, CNRS, Bordeaux INP, F-33405 Talence, France}
\email{jasmin.raissy@math.u-bordeaux.fr}

\thanks{This research was conducted during the Simons Symposium on Algebraic, Complex and Arithmetic Dynamics (2022).}

\thanks{The research of the first and fourth authors was supported by the ANR LabEx CIMI (grant ANR-11-LABX-0040) within the French State program ``Investissements d’Avenir''. The research of the second and fourth authors was supported by the Institut Universitaire de France (IUF). The research of the third author was supported by the German Research Foundation (DFG, project number 455038303).}


\begin{abstract}
Let $\mathcal{O}_{K}$ be the ring of integers of an imaginary quadratic field $K$. Recently, Ji and Xie proved that every rational map $f \colon \widehat{\mathbb{C}} \rightarrow \widehat{\mathbb{C}}$ of degree $d \geq 2$ whose multipliers all lie in $\mathcal{O}_{K}$ is a power map, a Chebyshev map or a Latt\`{e}s map. Their proof relies on a result from non-Archimedean dynamics obtained by Rivera-Letelier. In the present note, we show that one can avoid using this result by considering a differential equation instead. Our proof of Ji and Xie's result also applies to the case of entire maps. Thus, we also show that every nonaffine entire map $f \colon \mathbb{C} \rightarrow \mathbb{C}$ whose multipliers all lie in $\mathcal{O}_{K}$ is a power map or a Chebyshev map.
\end{abstract}

\maketitle

\section{Introduction}

Suppose that $S$ is a Riemann surface and $f \colon S \rightarrow S$ is a holomorphic map. We recall that a point $z_{0} \in S$ is \emph{periodic} for $f$ if there exists an integer $p \geq 1$ such that $f^{\circ p}\left( z_{0} \right) = z_{0}$. In this case, the least such integer $p$ is called the \emph{period} of $z_{0}$. The \emph{multiplier} of $f$ at $z_{0}$ is the unique eigenvalue $\lambda \in \mathbb{C}$ of the differential of $f^{\circ p}$ at $z_{0}$. By the chain rule, the multiplier is invariant under conjugation: if $\phi \colon S \rightarrow S$ is a biholomorphism and $g = \phi \circ f \circ \phi^{-1}$, then $\phi\left( z_{0} \right)$ is periodic for $g$ with period $p$ and multiplier $\lambda$. In this note, we will assume that $S$ represents either the complex line $\mathbb{C}$ or the Riemann sphere $\widehat{\mathbb{C}}$, and thus the map $f$ will be either entire or rational.

A rational map $f \colon \widehat{\mathbb{C}} \rightarrow \widehat{\mathbb{C}}$ of degree $d \geq 2$ is said to be
\begin{itemize}
\item a \emph{power map} if it is conjugate to $z \mapsto z^{\pm d}$,
\item a \emph{Chebyshev map} if it is conjugate to $\pm T_{d}$, where $T_{d}$ is the $d$th Chebyshev polynomial,
\item a \emph{Latt\`{e}s map} if there exist a $1$\nobreakdash-dimensional complex torus $\mathbb{T}$, a holomorphic map $L \colon \mathbb{T} \rightarrow \mathbb{T}$ and a nonconstant holomorphic map $p \colon \mathbb{T} \rightarrow \widehat{\mathbb{C}}$ that make the following diagram commute:
\begin{center}
\begin{tikzpicture}
\node (M00) at (0,0) {$\mathbb{T}$};
\node (M01) at (1.5,0) {$\mathbb{T}$};
\node (M10) at (0,-1.5) {$\widehat{\mathbb{C}}$};
\node (M11) at (1.5,-1.5) {$\widehat{\mathbb{C}}$};
\draw[myarrow] (M00) to node[above]{$L$} (M01);
\draw[myarrow] (M10) to node[below]{$f$} (M11);
\draw[myarrow] (M00) to node[left]{$p$} (M10);
\draw[myarrow] (M01) to node[right]{$p$} (M11);
\end{tikzpicture}
\end{center}
\end{itemize}
Power maps, Chebyshev maps and Latt\`{e}s maps are called finite quotients of affine maps by Milnor in~\cite{M2006} and exceptional maps by Ji and Xie in~\cite{JX2023}. In this note, we will use the second terminology.

As shown by Milnor in~\cite{M2006}, if $f \colon \widehat{\mathbb{C}} \rightarrow \widehat{\mathbb{C}}$ is a power map, a Chebyshev map or a Latt\`{e}s map, then its multipliers at its periodic points all belong to the ring of integers $\mathcal{O}_{K}$ of some imaginary quadratic field $K \subset \mathbb{C}$. Milnor conjectured that the converse is true. In~\cite{H2022}, the third author proved the conjecture for quadratic rational maps. Ji and Xie later proved the general case:

\begin{theorem}[{\cite[Theorem~1.12]{JX2023}}]
\label{theorem:rational}
Assume that $K \subset \mathbb{C}$ is an imaginary quadratic field and $f \colon \widehat{\mathbb{C}} \rightarrow \widehat{\mathbb{C}}$ is a rational map of degree $d \geq 2$ whose multipliers all lie in $\mathcal{O}_{K}$. Then $f$ is a power map, a Chebyshev map or a Latt\`{e}s map.
\end{theorem}

In this note, we present a variant of Ji and Xie's proof of Theorem~\ref{theorem:rational}. Our proof only differs from the original one in one of the arguments: where they use a result from non-Archimedean dynamics proved by Rivera-Letelier in~\cite{RL2003}, we consider a differential equation instead. Thus, our main contribution is Proposition~\ref{proposition:escaping}.

We also determine the entire maps with integer multipliers. A nonaffine entire map $f \colon \mathbb{C} \rightarrow \mathbb{C}$ is said to be a \emph{power map} or a \emph{Chebyshev map} if it is polynomial and it is a power map or a Chebyshev map in the previous sense. Equivalently, a nonaffine entire map $f \colon \mathbb{C} \rightarrow \mathbb{C}$ is
\begin{itemize}
\item a power map if it is conjugate to $z \mapsto z^{d}$ for some integer $d \geq 2$,
\item a Chebyshev map if it is conjugate to $\pm T_{d}$ for some integer $d \geq 2$.
\end{itemize}
Also note that Latt\`{e}s maps are not polynomial.

Our arguments to prove Theorem~\ref{theorem:rational} also apply to the case of entire maps. Thus, we obtain the result below, which shows that there is no transcendental entire map whose multipliers all lie in the ring of integers of some imaginary quadratic field.

\begin{theorem}
\label{theorem:entire}
Assume that $K \subset \mathbb{C}$ is an imaginary quadratic field and $f \colon \mathbb{C} \rightarrow \mathbb{C}$ is a nonaffine entire map whose multipliers all lie in $\mathcal{O}_{K}$. Then $f$ is a power map or a Chebyshev map.
\end{theorem}

\begin{remark}
In fact, our proof shows that the conclusions of Theorems~\ref{theorem:rational} and~\ref{theorem:entire} still hold if one only assumes that there is some open set $U$ that intersects the Julia set $\mathcal{J}_{f}$ of $f$ and such that the multipliers of $f$ at its periodic points in $U$ all lie in $\mathcal{O}_{K}$. This is also true of Ji and Xie's proof of Theorem~\ref{theorem:rational}.
\end{remark}

After writing this note, the third author obtained the following stronger version of Theorem~\ref{theorem:rational}:

\begin{theorem}[{\cite[Main Theorem]{H2023}}]
\label{theorem:huguin}
Assume that $K \subset \mathbb{C}$ is a number field and $f \colon \widehat{\mathbb{C}} \rightarrow \widehat{\mathbb{C}}$ is a rational map of degree $d \geq 2$ whose multipliers all lie in $K$. Then $f$ is a power map, a Chebyshev map or a Latt\`{e}s map.
\end{theorem}

Very recently, Ji, Xie and Zhang generalized Theorem~\ref{theorem:huguin}, proving the following:

\begin{theorem}[{\cite[Theorem~1.4]{JXZ2023}}]
Assume that $K \subset \mathbb{C}$ is a number field and $f \colon \widehat{\mathbb{C}} \rightarrow \widehat{\mathbb{C}}$ is a rational map of degree $d \geq 2$ such that, for every multiplier $\lambda$ of $f$, there is an integer $n \geq 1$ such that $\lvert \lambda \rvert^{n} \in K$. Then $f$ is a power map, a Chebyshev map or a Latt\`{e}s map.
\end{theorem}

In contrast, the first and third authors together with Gorbovickis later showed that Theorem~\ref{theorem:entire} does not generalize to the case of rational multipliers:

\begin{theorem}[{\cite[Theorem~4]{BGH2023}}]
Assume that $K \subset \mathbb{C}$ is a number field that is not contained in $\mathbb{R}$. Then there exist transcendental entire maps $f \colon \mathbb{C} \rightarrow \mathbb{C}$ whose multipliers all lie in $K$.
\end{theorem}

In Section~\ref{section:exceptional}, we present a characterization of power maps, Chebyshev maps and Latt\`{e}s maps. In Section~\ref{section:proof}, we present our proof of Theorems~\ref{theorem:rational} and~\ref{theorem:entire}.

\begin{acknowledgments}
The authors would like to thank the anonymous referee for his valuable comments, and in particular for simplifying our proof and suggesting that our arguments might also apply to the case of entire maps.
\end{acknowledgments}

\section{Exceptional maps and escaping quadratic-like maps}
\label{section:exceptional}

Throughout this section, we assume that $S$ denotes either $\mathbb{C}$ or $\widehat{\mathbb{C}}$. We say that a holomorphic map $f \colon S \rightarrow S$ is \emph{nonlinear} if it is neither constant nor injective. In other words, the nonlinear holomorphic maps $f \colon S \rightarrow S$ are precisely the nonaffine entire maps if $S = \mathbb{C}$ and the rational maps of degree $d \geq 2$ if $S = \widehat{\mathbb{C}}$.

\subsection{Exceptional maps}

We say that a nonlinear holomorphic map $f \colon S \rightarrow S$ is \emph{exceptional} if it is a power map, a Chebyshev map or a Latt\`{e}s map.

Ritt obtained the following characterization of exceptional maps:

\begin{lemma}[\cite{R1922}]
\label{lemma:ritt}
Suppose that $f \colon S \rightarrow S$ is a nonlinear holomorphic map, $\phi \colon \mathbb{C} \rightarrow S$ is a nonconstant holomorphic map, $\alpha \colon \mathbb{C} \rightarrow \mathbb{C}$ is an affine map that is not a translation and $\tau \colon \mathbb{C} \rightarrow \mathbb{C}$ is a nontrivial translation such that \[ \phi \circ \alpha = f \circ \phi \quad \text{and} \quad \phi \circ \tau = \phi \, \text{.} \] Then $f$ is exceptional.
\end{lemma}

\begin{remark}
In fact, Ritt is interested in the equation $\phi \circ \alpha = f \circ \phi$, with $f \colon \widehat{\mathbb{C}} \rightarrow \widehat{\mathbb{C}}$ a rational map, $\alpha \colon \mathbb{C} \rightarrow \mathbb{C}$ a nontrivial homothety about the origin and $\phi \colon \mathbb{C} \rightarrow \widehat{\mathbb{C}}$ a periodic and nonconstant holomorphic map, and he seeks to find $\phi$. As he points out in the last sentence of his introduction, his arguments also apply if $f \colon \mathbb{C} \rightarrow \mathbb{C}$ is an entire map and $\phi \colon \mathbb{C} \rightarrow \mathbb{C}$ is a periodic and nonconstant entire map. We may also take $\alpha \colon \mathbb{C} \rightarrow \mathbb{C}$ to be any affine map that is not a translation, conjugating $\alpha$ if necessary to reduce to the previous situation. Finally, one easily deduces from the form of $\phi$ found by Ritt that $f$ is necessarily exceptional if it is nonlinear.
\end{remark}

The following generalization of Lemma~\ref{lemma:ritt} is essentially due to Ji and Xie (compare~\cite[Lemma~2.9]{JX2023}).

\begin{lemma}
\label{lemma:rittBis}
Suppose that $f \colon S \rightarrow S$ is a nonlinear holomorphic map, $\phi \colon \mathbb{C} \rightarrow S$ is a nonconstant holomorphic map and $\alpha_{1} \colon \mathbb{C} \rightarrow \mathbb{C}$ and $\alpha_{2} \colon \mathbb{C} \rightarrow \mathbb{C}$ are affine maps that do not commute and such that \[ \phi \circ \alpha_{1} = f \circ \phi = \phi \circ \alpha_{2} \, \text{.} \] Then $f$ is exceptional.
\end{lemma}

\begin{proof}
Note that $\alpha_{1}$ or $\alpha_{2}$ is not a translation as, otherwise, they would commute. Also note that $\alpha_{1}$ and $\alpha_{2}$ are both nonconstant because $f$ and $\phi$ are not constant. Define the affine map \[ \tau = \alpha_{1} \circ \alpha_{2}^{-1} \circ \alpha_{1} \circ \alpha_{2} \circ \left( \alpha_{1}^{-1} \right)^{\circ 2} \, \text{.} \] Then $\tau$ is a translation as the linear endomorphism associated with a composition of affine maps equals the composition of the associated linear endomorphisms and linear endomorphisms of $\mathbb{C}$ commute. Also note that $\tau$ is not the identity map as, otherwise, $\alpha_{1}$ and $\alpha_{2}$ would commute. Thus, $\tau \colon \mathbb{C} \rightarrow \mathbb{C}$ is a nontrivial translation. Moreover, we have \[ \phi \circ \alpha_{1} \circ \alpha_{2} = f \circ \phi \circ \alpha_{2} = f^{\circ 2} \circ \phi = f \circ \phi \circ \alpha_{1} = \phi \circ \alpha_{1}^{\circ 2} \, \text{,} \] and hence \[ \phi \circ \tau = \phi \circ \alpha_{1} \circ \alpha_{2}^{-1} \circ \alpha_{1} \circ \alpha_{2} \circ \left( \alpha_{1}^{-1} \right)^{\circ 2} = \phi \circ \alpha_{1} \circ \alpha_{2} \circ \left( \alpha_{1}^{-1} \right)^{\circ 2} = \phi \, \text{.} \] Therefore, $f$ is exceptional by Lemma~\ref{lemma:ritt}, and the lemma is proved.
\end{proof}

\subsection{Escaping quadratic-like maps}

An \emph{escaping quadratic-like map} is a holomorphic covering map $f \colon U \rightarrow V$ of degree $2$, with $U, V$ nonempty open subsets of $\mathbb{C}$ such that $U \Subset V$ and $V$ is simply connected. In this situation, note that $U$ has two connected components $U_{1}$ and $U_{2}$, which are mapped biholomorphically onto $V$ by $f$.

\begin{remark}
The notion of escaping quadratic-like map is related to the well-known one of quadratic-like map as follows: Recall that a \emph{quadratic-like map} is a proper holomorphic map $f \colon V \rightarrow W$ of degree $2$, with $V \Subset W$ nonempty simply connected open subsets of $\mathbb{C}$. In this situation, $f \colon V \rightarrow W$ has a unique critical point $\gamma \in V$. If $f(\gamma) \in W \setminus V$, then $f \colon f^{-1}(V) \rightarrow f^{-1}(W)$ is an escaping quadratic-like map.
\end{remark}

We shall use the result below, which was proved by Bergweiler. His proof relies on a weak version of the Ahlfors five islands theorem. Here, we provide a proof in the particular case of rational maps, which follows Ji and Xie's proof of Theorem~\ref{theorem:rational} (compare~\cite[Proposition~2.6]{JX2023}). We refer the reader to Bergweiler's article for a proof of the general case.

\begin{lemma}[{\cite[Proposition~B.3]{B2000}}]
\label{lemma:restriction}
Suppose that $f \colon S \rightarrow S$ is a nonlinear holomorphic map. Then there exist an integer $n \geq 1$ and open subsets $U, V$ of $\mathbb{C}$ such that $f^{\circ n} \colon U \rightarrow V$ is an escaping quadratic-like map.
\end{lemma}

\begin{proof}[Proof in the case of rational maps]
Suppose here that $f \colon \widehat{\mathbb{C}} \rightarrow \widehat{\mathbb{C}}$ is a rational map of degree $d \geq 2$. Then $f$ has infinitely many repelling periodic points. Moreover, $f$ has only finitely many critical points and, hence, also only finitely many periodic points that lie in the forward orbit of a critical point. Therefore, $f$ has a repelling periodic point $z_{1} \in \mathbb{C}$, with period $p \geq 1$, that does not lie in the forward orbit of a critical point. There exist a simply connected open neighborhood $V_{1}$ of $z_{1}$ in $\mathbb{C}$ and a local inverse $g_{1} \colon V_{1} \rightarrow g_{1}\left( V_{1} \right)$ of $f^{\circ p}$ such that $g_{1}\left( z_{1} \right) = z_{1}$ and $g_{1}\left( V_{1} \right) \Subset V_{1}$. Now, $z_{1}$ lies in the Julia set $\mathcal{J}_{f}$ of $f$ and its iterated preimages accumulate on all of $\mathcal{J}_{f}$, and hence there exist $\ell \geq 1$ and $z_{2} \in V_{1} \setminus \left\lbrace z_{1} \right\rbrace$ such that $f^{\circ \ell}\left( z_{2} \right) = z_{1}$. The point $z_{2}$ is not critical for $f$ by the definition of $z_{1}$, and hence there exist a simply connected open neighborhood $V \subset V_{1}$ of $z_{1}$ and a local inverse $g_{2} \colon V \rightarrow g_{2}(V)$ of $f^{\circ \ell}$ such that $g_{2}\left( z_{1} \right) = z_{2}$ and $g_{2}(V) \Subset V_{1} \setminus \left\lbrace z_{1} \right\rbrace$. Now, note that $g_{1} \colon V_{1} \rightarrow g_{1}\left( V_{1} \right)$ is a contracting map with respect to the Poincar\'{e} metric on $V_{1}$ since $g_{1}\left( V_{1} \right) \Subset V_{1}$. Therefore, as $g_{1}\left( z_{1} \right) = z_{1}$, $g_{2}(V) \Subset V_{1} \setminus \left\lbrace z_{1} \right\rbrace$ and $V \subset V_{1}$, there exist $m_{1}, m_{2} \geq 1$ such that \[ W_{2} = g_{1}^{\circ m_{1}} \circ g_{2}(V) \Subset V \quad \text{and} \quad W_{1} = g_{1}^{\circ m_{2}}(V) \Subset V \setminus W_{2} \, \text{.} \] Define \[ h_{1} = g_{1}^{\circ m_{2}} \colon V \rightarrow W_{1} \quad \text{and} \quad h_{2} = g_{1}^{\circ m_{1}} \circ g_{2} \colon V \rightarrow W_{2} \, \text{.} \] Note that $h_{1}$ is a local inverse of $f^{\circ n_{1}}$, with $n_{1} = m_{2} p$, and $h_{2}$ is a local inverse of $f^{\circ n_{2}}$, with $n_{2} = m_{1} p +\ell$. Set \[ n = n_{1} n_{2} \quad \text{and} \quad U = h_{1}^{\circ n_{2}}(V) \cup h_{2}^{\circ n_{1}}(V) \, \text{.} \] Then $f^{\circ n} \colon U \rightarrow V$ is an escaping quadratic-like map. This completes the proof of the lemma in the case of rational maps.
\end{proof}

\subsection{Affine escaping quadratic-like maps}

We say that an escaping quadratic-like map $f \colon U \rightarrow V$ is \emph{affine} if it is affine on each of the two connected components of $U$. We say that two escaping quadratic-like maps $f_{1} \colon U_{1} \rightarrow V_{1}$ and $f_{2} \colon U_{2} \rightarrow V_{2}$ are \emph{conjugate} if there exists a biholomorphism $\phi \colon V_{1} \rightarrow V_{2}$ such that $\phi \circ f_{1} = f_{2} \circ \phi$ on $U_{1}$.

In our proof of Theorems~\ref{theorem:rational} and~\ref{theorem:entire}, we shall use the following characterization of exceptional maps (compare~\cite[Theorem~1.1]{JX2023}):

\begin{lemma}
\label{lemma:characterization}
Suppose that $f \colon S \rightarrow S$ is a nonlinear holomorphic map such that an escaping quadratic-like map of the form $f^{\circ n} \colon U \rightarrow V$, with $n \geq 1$ and $U, V \subset \mathbb{C}$, is conjugate to an affine escaping quadratic-like map. Then $f$ is exceptional.
\end{lemma}

\begin{proof}
By hypothesis, there exist an affine escaping quadratic-like map $g \colon U^{\prime} \rightarrow V^{\prime}$ and a biholomorphism $\phi \colon V^{\prime} \rightarrow V$ such that $\phi \circ g = f^{\circ n} \circ \phi$ on $U^{\prime}$. Denote by $U_{1}^{\prime}$ and $U_{2}^{\prime}$ the two connected components of $U^{\prime}$. Then the restrictions of $g$ to $U_{1}^{\prime}$ and $U_{2}^{\prime}$ agree with the restrictions of affine maps $\alpha_{1} \colon \mathbb{C} \rightarrow \mathbb{C}$ and $\alpha_{2} \colon \mathbb{C} \rightarrow \mathbb{C}$. We have $\phi \circ \alpha_{1} = f^{\circ n} \circ \phi$ on $U_{1}^{\prime}$ and $\phi \circ \alpha_{2} = f^{\circ n} \circ \phi$ on $U_{2}^{\prime}$. As the affine maps $\alpha_{1}$ and $\alpha_{2}$ are repelling, we may use any of these two relations to extend $\phi$ to a holomorphic map $\widehat{\phi} \colon \mathbb{C} \rightarrow S$. By the identity principle, we have \[ \widehat{\phi} \circ \alpha_{1} = f^{\circ n} \circ \widehat{\phi} = \widehat{\phi} \circ \alpha_{2} \, \text{.} \] Furthermore, $\alpha_{1}$ and $\alpha_{2}$ have distinct fixed points, in $U_{1}^{\prime}$ and $U_{2}^{\prime}$ respectively, and hence they do not commute. Therefore, the map $f^{\circ n}$ is exceptional by Lemma~\ref{lemma:rittBis}, and hence so is $f$. Thus, the lemma is proved.
\end{proof}

\begin{remark}
The proof above uses the fact that, if $f \colon \widehat{\mathbb{C}} \rightarrow \widehat{\mathbb{C}}$ is a rational map of degree $d \geq 2$ such that $f^{\circ n}$ is exceptional for some $n \geq 1$, then so is $f$. This can be deduced from the following facts. Any postcritically finite rational map $f \colon \widehat{\mathbb{C}} \rightarrow \widehat{\mathbb{C}}$ of degree $d \geq 2$ has an associated orbifold $O_{f}$. A rational map $f \colon \widehat{\mathbb{C}} \rightarrow \widehat{\mathbb{C}}$ of degree $d \geq 2$ is an exceptional map if and only if it is postcritically finite and its orbifold $O_{f}$ is parabolic. If $f \colon \widehat{\mathbb{C}} \rightarrow \widehat{\mathbb{C}}$ is a rational map of degree $d \geq 2$ and $n \geq 1$, then $f$ is postcritically finite if and only if $f^{\circ n}$ is postcritically finite and, in this case, we have $O_{f} = O_{f^{\circ n}}$. We refer the reader to~\cite{BM2017} for further details.
\end{remark}

\section{Proof of the results}
\label{section:proof}

As for entire and rational maps, we can define the notions of periodic point and multiplier for escaping quadratic-like maps. It follows from Lemmas~\ref{lemma:restriction} and~\ref{lemma:characterization} that Theorems~\ref{theorem:rational} and~\ref{theorem:entire} are a consequence of the following result:

\begin{proposition}
\label{proposition:escaping}
Assume that $K$ is an imaginary quadratic field and $f \colon U \rightarrow V$ is an escaping quadratic-like map whose multipliers all lie in $\mathcal{O}_{K}$. Then $f$ is conjugate to an affine escaping quadratic-like map.
\end{proposition}

Our proof of this statement will occupy the rest of the note. Assume from now on that $f \colon U \rightarrow V$ is an escaping quadratic-like map whose multipliers all lie in the ring of integers $\mathcal{O}_{K}$ of some imaginary quadratic field $K$. Denote by $U_{1}$ and $U_{2}$ the two connected components of $U$ and define \[ f_{1} = f \vert_{U_{1}} \, \text{,} \quad f_{2} = f \vert_{U_{2}} \, \text{,} \quad g_{1} = f_{1}^{-1} \colon V \rightarrow U_{1} \, \text{,} \quad g_{2} = f_{2}^{-1} \colon V \rightarrow U_{2} \, \text{.} \] As $U \Subset V$, the maps $g_{1}$ and $g_{2}$ are contracting with respect to the Poincar\'{e} metric on $V$. Therefore, the maps $f_{1}$ and $f_{2}$ have unique fixed points $z_{1} \in U_{1}$ and $z_{2} \in U_{2}$ respectively, which are repelling. Denote by $\lambda_{1}$ and $\lambda_{2}$ their associated multipliers, which lie in $\mathcal{O}_{K}$.

By the Koenigs linearization theorem, the sequence $\left( \phi_{n} \right)_{n \geq 0}$ of univalent maps defined on $V$ by \[ \phi_{n}(z) = \lambda_{1}^{n} \left( g_{1}^{\circ n}(z) -z_{1} \right) \] converges to a univalent map $\phi \colon V \rightarrow \mathbb{C}$ such that $\phi \circ f_{1} = \lambda_{1} \phi$. Thus, replacing $f$ by $\phi \circ f \circ \phi^{-1}$ if necessary, we may assume that $f_{1}(z) = \lambda_{1} z$ for all $z \in U_{1}$, which yields $U_{1} = \frac{1}{\lambda_{1}} V$ and $z_{1} = 0$. We shall prove that $f_{2}$ is affine.

\subsection{A special sequence of periodic points}

Still following Ji and Xie's proof, we consider a particular sequence of periodic points for $f$. For each $n \geq 1$, the map $g_{2} \circ g_{1}^{\circ (n -1)} \colon V \rightarrow U_{2}$ is contracting with respect to the Poincar\'{e} metric on $V$, and hence it has a unique fixed point $w_{n} \in U_{2}$. For every $n \geq 1$, we have \[ \forall j \in \lbrace 1, \dotsc, n \rbrace, \, f^{\circ j}\left( w_{n} \right) = g_{1}^{\circ (n -j)}\left( w_{n} \right) = \frac{w_{n}}{\lambda_{1}^{n -j}} \, \text{.} \] In particular, for every $n \geq 1$, the point $w_{n}$ is periodic for $f$ with period $n$ and its associated multiplier $\rho_{n}$ satisfies \[ \rho_{n} = \prod_{j = 1}^{n} f^{\prime}\left( \frac{w_{n}}{\lambda_{1}^{n -j}} \right) = \lambda_{1}^{n -1} f_{2}^{\prime}\left( w_{n} \right) \in \mathcal{O}_{K} \, \text{.} \] Note that \[ w_{n} = g_{2}\left( \frac{w_{n}}{\lambda_{1}^{n -1}} \right) = \alpha +\frac{\beta}{\lambda_{1}^{n -1}} +o\left( \frac{1}{\lambda_{1}^{n}} \right) \quad \text{as} \quad n \rightarrow +\infty \, \text{,} \quad \text{with} \quad \left\lbrace \begin{array}{@{} l @{}} \alpha = g_{2}(0)\\ \beta = \alpha g_{2}^{\prime}(0) \end{array} \right. \, \text{,} \] since $\lim\limits_{n \rightarrow +\infty} w_{n} = \alpha$, and hence \[ \rho_{n} = \lambda_{1}^{n -1} f_{2}^{\prime}\left( w_{n} \right) = a \lambda_{1}^{n -1} +b +o(1) \quad \text{as} \quad n \rightarrow +\infty \, \text{,} \quad \text{with} \quad \left\lbrace \begin{array}{@{} l @{}} a = f_{2}^{\prime}(\alpha)\\ b = \beta f_{2}^{\prime \prime}(\alpha) \end{array} \right. \, \text{.} \]

Now, let us use the assumption that the multipliers of $f$ all lie in $\mathcal{O}_{K}$ to obtain the following (see~\cite[Lemma~3.1]{JX2023}):

\begin{claim}
We have $\rho_{n} = a \lambda_{1}^{n -1} +b$ for all $n$ sufficiently large.
\end{claim}

\begin{proof}
Write $\rho_{n} = a \lambda_{1}^{n -1} +b +\varepsilon_{n}$ for $n \geq 1$, so that $\lim\limits_{n \rightarrow +\infty} \varepsilon_{n} = 0$. Then, for every $n \geq 1$, we have \[ \lambda_{1} \rho_{n} -\rho_{n +1} = \left( \lambda_{1} -1 \right) b +\lambda_{1} \varepsilon_{n} -\varepsilon_{n +1} \in \mathcal{O}_{K} \, \text{.} \] Moreover, $\lim\limits_{n \rightarrow +\infty} \left( \lambda_{1} \varepsilon_{n} -\varepsilon_{n +1} \right) = 0$. It follows that $\left( \lambda_{1} -1 \right) b \in \mathcal{O}_{K}$ because $\mathcal{O}_{K}$ is closed in $\mathbb{C}$. Therefore, since $\mathcal{O}_{K}$ is discrete, for every $n$ sufficiently large, we have $\lambda_{1} \rho_{n} -\rho_{n +1} = \left( \lambda_{1} -1 \right) b$, and hence $\varepsilon_{n +1} = \lambda_{1} \varepsilon_{n}$. As $\left\lvert \lambda_{1} \right\rvert > 1$ and $\lim\limits_{n \rightarrow +\infty} \varepsilon_{n} = 0$, this yields $\varepsilon_{n} = 0$ for all $n$ sufficiently large. Thus, the claim is proved.
\end{proof}

\subsection{A differential equation}

Our proof now deviates from Ji and Xie's proof of Theorem~\ref{theorem:rational}. Instead of using a result concerning the non-Archimedean dynamics of rational maps, we show that $f_{2}$ is solution of a simple differential equation.

\begin{claim}
\label{claim:differential}
The holomorphic map $f_{2} \colon U_{2} \rightarrow V$ satisfies
\begin{equation}
\label{equation:differential}\tag{E}
\forall z \in U_{2}, \, f_{2}^{\prime}(z) = a +b \frac{f_{2}(z)}{z} \, \text{.}
\end{equation}
\end{claim}

\begin{proof}
For every $n$ sufficiently large, we have \[ f_{2}^{\prime}\left( w_{n} \right) = \frac{\rho_{n}}{\lambda_{1}^{n -1}} = a +\frac{b}{\lambda_{1}^{n -1}} = a +b \frac{f_{2}\left( w_{n} \right)}{w_{n}} \, \text{.} \] Therefore, since $\left( w_{n} \right)_{n \geq 1}$ accumulates at $\alpha \in U_{2}$, the relation \eqref{equation:differential} follows from the identity principle. Thus, the claim is proved.
\end{proof}

Note that, as \eqref{equation:differential} is a first-order linear ordinary differential equation, it may be easily solved. However, we shall not use the explicit form of the solutions.

\subsection{Conclusion}

To conclude, note that $f^{\circ 2} \colon \frac{1}{\lambda_{1}} U_{1} \cup g_{2}\left( U_{2} \right) \rightarrow V$ is an escaping quadratic-like map such that $f^{\circ 2}(z) = \lambda_{1}^{2} z$ for all $z \in \frac{1}{\lambda_{1}} U_{1}$ and whose multipliers all lie in $\mathcal{O}_{K}$. Applying Claim~\ref{claim:differential} with $f^{\circ 2}$ instead of $f$, it follows that there exist $\widehat{a}, \widehat{b} \in \mathbb{C}$ such that \[ \forall z \in g_{2}\left( U_{2} \right), \, \left( f_{2}^{\circ 2} \right)^{\prime}(z) = \widehat{a} +\widehat{b} \frac{f_{2}^{\circ 2}(z)}{z} \, \text{.} \] Moreover, by the relation~\eqref{equation:differential}, we have \[ \forall z \in g_{2}\left( U_{2} \right), \, \left( f_{2}^{\circ 2} \right)^{\prime}(z) = f_{2}^{\prime}(z) \cdot f_{2}^{\prime}\left( f_{2}(z) \right) = \left( a +b \frac{f_{2}(z)}{z} \right) \left( a +b \frac{f_{2}^{\circ 2}(z)}{f_{2}(z)} \right) \, \text{.} \] Thus, combining the two relations above, we obtain
\begin{equation}
\label{equation:iterate}\tag{$\diamondsuit$}
\forall z \in g_{2}\left( U_{2} \right), \, a^{2} +a b \frac{f_{2}(z)}{z} +a b \frac{f_{2}^{\circ 2}(z)}{f_{2}(z)} +b^{2} \frac{f_{2}^{\circ 2}(z)}{z} = \widehat{a} +\widehat{b} \frac{f_{2}^{\circ 2}(z)}{z} \, \text{.}
\end{equation}

Let us prove that $\widehat{b} = b^{2}$. To obtain a contradiction, assume that $\widehat{b} \neq b^{2}$. Then, by the relation~\eqref{equation:iterate}, we have \[ \forall z \in g_{2}\left( U_{2} \right), \, H(z) \cdot H\left( f_{2}(z) \right) = A^{2} +\widehat{a} -a^{2} \, \text{,} \] where $A, B \in \mathbb{C}$ and $H \colon U_{2} \rightarrow \mathbb{C}$ are defined by \[ B^{2} = b^{2} -\widehat{b} \, \text{,} \quad A B = a b \quad \text{and} \quad H \colon z \mapsto A +B \frac{f_{2}(z)}{z} \, \text{.} \] It follows that \[ \forall z \in g_{2}^{\circ 2}\left( U_{2} \right), \, H\left( f_{2}(z) \right) \cdot H\left( f_{2}^{\circ 2}(z) \right) = H(z) \cdot H\left( f_{2}(z) \right) \, \text{,} \] which yields \[ \forall z \in g_{2}^{\circ 2}\left( U_{2} \right), \, H\left( f_{2}^{\circ 2}(z) \right) = H(z) \, \text{.} \] Therefore, differentiating the relation above and evaluating at $z = z_{2}$, we have \[ \frac{B \left( \lambda_{2} -1 \right)^{2} \left( \lambda_{2} +1 \right)}{z_{2}} = \left( \lambda_{2}^{2} -1 \right) H^{\prime}\left( z_{2} \right) = 0 \, \text{,} \] which is a contradiction since $B \neq 0$ and $\left\lvert \lambda_{2} \right\rvert > 1$. Thus, $\widehat{b} = b^{2}$.

By the relation~\eqref{equation:iterate}, it follows that \[ \forall z \in g_{2}\left( U_{2} \right), \, a b \frac{f_{2}(z)}{z} +a b \frac{f_{2}^{\circ 2}(z)}{f_{2}(z)} = \widehat{a} -a^{2} \, \text{.} \] Differentiating the relation above and evaluating at $z = z_{2}$, we obtain \[ \frac{a b \left( \lambda_{2} -1 \right) \left( \lambda_{2} +1 \right)}{z_{2}} = 0 \, \text{,} \] which yields $b = 0$ since $a \neq 0$ and $\left\lvert \lambda_{2} \right\rvert > 1$. Therefore, $f_{2}^{\prime}(z) = a$ for all $z \in U_{2}$ by the relation~\eqref{equation:differential}, and thus $f_{2}$ is affine. This completes the proof of Proposition~\ref{proposition:escaping}.

\providecommand{\bysame}{\leavevmode\hbox to3em{\hrulefill}\thinspace}
\providecommand{\MR}{\relax\ifhmode\unskip\space\fi MR }
\providecommand{\MRhref}[2]{%
	\href{http://www.ams.org/mathscinet-getitem?mr=#1}{#2}
}
\providecommand{\href}[2]{#2}


\begin{thebibliography}{BGH23}

\bibitem[Ber00]{B2000}
Walter Bergweiler, \emph{The role of the {A}hlfors five islands theorem in
	complex dynamics}, Conform. Geom. Dyn. \textbf{4} (2000), 22--34.
\MR{1741773}

\bibitem[BGH23]{BGH2023}
Xavier Buff, Igors Gorbovickis, and Valentin Huguin, \emph{Entire maps with
	rational preperiodic points and multipliers}, arXiv:2304.13674v1, 2023.

\bibitem[BM17]{BM2017}
Mario Bonk and Daniel Meyer, \emph{Expanding {T}hurston maps}, Mathematical
Surveys and Monographs, vol. 225, American Mathematical Society, Providence,
RI, 2017. \MR{3727134}

\bibitem[Hug22]{H2022}
Valentin Huguin, \emph{Quadratic rational maps with integer multipliers}, Math.
Z. \textbf{302} (2022), no.~2, 949--969. \MR{4480217}

\bibitem[Hug23]{H2023}
\bysame, \emph{Rational maps with rational multipliers}, J. \'{E}c. polytech.
Math. \textbf{10} (2023), 591--599. \MR{4573898}

\bibitem[JX23]{JX2023}
Zhuchao Ji and Junyi Xie, \emph{Homoclinic orbits, multiplier spectrum and
	rigidity theorems in complex dynamics}, Forum Math. Pi \textbf{11} (2023),
Paper No. e11, 37. \MR{4585467}

\bibitem[JXZ23]{JXZ2023}
Zhuchao Ji, Junyi Xie, and Geng-Rui Zhang, \emph{Space spanned by
	characteristic exponents}, arXiv:2308.00289v1, 2023.

\bibitem[Mil06]{M2006}
John Milnor, \emph{On {L}att\`{e}s maps}, Dynamics on the {R}iemann sphere,
Eur. Math. Soc., Z\"{u}rich, 2006, pp.~9--43. \MR{2348953}

\bibitem[Rit22]{R1922}
J.~F. Ritt, \emph{Periodic functions with a multiplication theorem}, Trans.
Amer. Math. Soc. \textbf{23} (1922), no.~1, 16--25. \MR{1501186}

\bibitem[RL03]{RL2003}
Juan Rivera-Letelier, \emph{Dynamique des fonctions rationnelles sur des corps
	locaux}, Ast\'{e}risque (2003), no.~287, xv, 147--230. \MR{2040006}

\end{thebibliography}
\end{document}